\documentclass[12pt,leqno]{article}
\usepackage{amsfonts}
\usepackage{amsmath, amsthm, amsfonts, amssymb, color}
\usepackage{mathrsfs}
\usepackage{color}
\newtheorem{thm}{Theorem}[section]

\newtheorem{lem}[thm]{Lemma}

\newtheorem{exa}[thm]{Example}
\newtheorem{rem}[thm]{Remark}
\theoremstyle{definition}

\newcommand{\scr}[1]{\mathscr #1}
\definecolor{wco}{rgb}{0.5,0.2,0.3}

\numberwithin{equation}{section} \theoremstyle{remark}

\newcommand{\ua}{\uparrow}

\title{{\bf  Bismut Formula for Intrinsic/Lions Derivatives of Distribution Dependent SDEs with Singular Coefficients}\footnote{X. Huang is supported by NSFC (No.11801406), Y. Song is supported by NSFC (No.11971227, 11790272) and F.-Y. Wang is supported by NSFC (No.11771326, 11831014, 11921001)} }
\author{
{\bf   Xing Huang $^{a)}$,  Yulin Song $^{b)}$, Feng-Yu Wang $^{a), c)}$  }\\
\footnotesize{ a)Center for Applied Mathematics, Tianjin
University, Tianjin, 300072, China}\\
 \footnotesize{ b)Department of Mathematics,
Nanjing University, Nanjing, 210093, China}\\
\footnotesize{ c)Department of Mathematics,
Swansea University, Bay Campus, Swansea, SA1 8EN, United Kingdom}\\
\footnotesize{xinghuang@tju.edu.cn, \ ylsong@nju.edu.cn,\  wangfy@tju.edu.cn}}
\begin{document}
\allowdisplaybreaks
\def\R{\mathbb R}  \def\ff{\frac} \def\ss{\sqrt} \def\B{\mathbf
B} \def\W{\mathbb W}
\def\N{\mathbb N} \def\kk{\kappa} \def\m{{\bf m}}
\def\ee{\varepsilon}\def\ddd{D^*}
\def\dd{\delta} \def\DD{\Delta} \def\vv{\varepsilon} \def\rr{\rho}
\def\<{\langle} \def\>{\rangle} \def\GG{\Gamma} \def\gg{\gamma}
  \def\nn{\nabla} \def\pp{\partial} \def\E{\mathbb E}
\def\d{\text{\rm{d}}} \def\bb{\beta} \def\aa{\alpha} \def\D{\scr D}
  \def\si{\sigma} \def\ess{\text{\rm{ess}}}
\def\beg{\begin} \def\beq{\begin{equation}}  \def\F{\scr F}
\def\Ric{\text{\rm{Ric}}} \def\Hess{\text{\rm{Hess}}}
\def\e{\text{\rm{e}}} \def\ua{\underline a} \def\OO{\Omega}  \def\oo{\omega}
 \def\tt{\tilde} \def\Ric{\text{\rm{Ric}}}
\def\cut{\text{\rm{cut}}} \def\P{\mathbb P} \def\ifn{I_n(f^{\bigotimes n})}
\def\C{\scr C}      \def\aaa{\mathbf{r}}     \def\r{r}
\def\gap{\text{\rm{gap}}} \def\prr{\pi_{{\bf m},\varrho}}  \def\r{\mathbf r}
\def\Z{\mathbb Z} \def\vrr{\varrho} \def\ll{\lambda}
\def\L{\scr L}\def\Tt{\tt} \def\TT{\tt}\def\II{\mathbb I}
\def\i{{\rm in}}\def\Sect{{\rm Sect}}  \def\H{\mathbb H}
\def\M{\scr M}\def\Q{\mathbb Q} \def\texto{\text{o}} \def\LL{\Lambda}
\def\Rank{{\rm Rank}} \def\B{\scr B} \def\i{{\rm i}} \def\HR{\hat{\R}^d}
\def\to{\rightarrow}\def\l{\ell}\def\iint{\int}
\def\EE{\scr E}\def\Cut{{\rm Cut}}
\def\A{\scr A} \def\Lip{{\rm Lip}}
\def\BB{\scr B}\def\Ent{{\rm Ent}}\def\L{\scr L}
\def\R{\mathbb R}  \def\ff{\frac} \def\ss{\sqrt} \def\B{\mathbf
B}
\def\N{\mathbb N} \def\kk{\kappa} \def\m{{\bf m}}
\def\dd{\delta} \def\DD{\Delta} \def\vv{\varepsilon} \def\rr{\rho}
\def\<{\langle} \def\>{\rangle} \def\GG{\Gamma} \def\gg{\gamma}
  \def\nn{\nabla} \def\pp{\partial} \def\E{\mathbb E}
\def\d{\text{\rm{d}}} \def\bb{\beta} \def\aa{\alpha} \def\D{\scr D}
  \def\si{\sigma} \def\ess{\text{\rm{ess}}}
\def\beg{\begin} \def\beq{\begin{equation}}  \def\F{\scr F}
\def\Ric{\text{\rm{Ric}}} \def\Hess{\text{\rm{Hess}}}
\def\e{\text{\rm{e}}} \def\ua{\underline a} \def\OO{\Omega}  \def\oo{\omega}
 \def\tt{\tilde} \def\Ric{\text{\rm{Ric}}}
\def\cut{\text{\rm{cut}}} \def\P{\mathbb P} \def\ifn{I_n(f^{\bigotimes n})}
\def\C{\scr C}      \def\aaa{\mathbf{r}}     \def\r{r}
\def\gap{\text{\rm{gap}}} \def\prr{\pi_{{\bf m},\varrho}}  \def\r{\mathbf r}
\def\Z{\mathbb Z} \def\vrr{\varrho} \def\ll{\lambda}
\def\L{\scr L}\def\Tt{\tt} \def\TT{\tt}\def\II{\mathbb I}
\def\i{{\rm in}}\def\Sect{{\rm Sect}}  \def\H{\mathbb H}
\def\M{\scr M}\def\Q{\mathbb Q} \def\texto{\text{o}} \def\LL{\Lambda}
\def\Rank{{\rm Rank}} \def\B{\scr B} \def\i{{\rm i}} \def\HR{\hat{\R}^d}
\def\to{\rightarrow}\def\l{\ell}
\def\8{\infty}\def\I{1}\def\U{\scr U}
\maketitle

\begin{abstract} By using distribution dependent Zvonkin's transforms and Malliavin calculus,  the Bismut type formula is derived for the intrinisc/Lions derivatives of distribution dependent SDEs with singular drifts, which generalizes the corresponding results derived for classical SDEs and regular distribution dependent SDEs.

\end{abstract} \noindent
 \vskip 2cm

\section{Introduction}

Due to wide  applications in the study of nonlinear PDEs and particle systems,  distribution dependent  stochastic differential equations (DDSDEs for short),
also called McKean-Vlasov or mean-field SDEs, have been intensively investigated,    see for instance  \cite{BT, CD, Carmona,BR1,BR2,HRW, BRTV, CN, Dawson, 20HRW, HW, HW20a, K, MV, SZ} among many other references.

 To characterize the regularity of DDSDEs, Bismut formula and derivative estimates have been presented for the distribution of solutions with respect to initial data, see  for instance  \cite{FYW1,B,CM,RW,BRW,S2020}. See also \cite{CN} for  the   study of   decoupled SDEs where the distribution parameter is fixed as the law of the associated DDSDE, and the resulting regularity estimates  apply to the DDSDEs as well (see Remark 2.2 below).

In this paper, we aim to establish Bismut formula for the Lions derivative of singular DDSDEs, such that existing results derived in more regular situations are extended.
This type  formula was first found  by Bismut \cite{Bismut84} in 1984
using Malliavin calculus for diffusion semigroups on manifolds, then reproved by Elworthy-Li \cite{El94} in 1974 using martingale arguments.
Since then the formula has been widely developed and  applied for SDEs/SPDEs driven by Gaussian or L\'evy noises.  Recently, Bismut formula was established in \cite{XXZZ} for SDEs with singular drifts by using Zvonkin's transform \cite{AZ}, which is a powerful tool in regularizing singular SDEs. In this paper, we aim to extend this result for singular DDSDEs.

Let $\scr P$ be the set of all probability measures on $\R^d$. Consider the following distribution-dependent SDE on $\R^d$:
\beq\label{E1}
\d X_t= (B_t+b_t)(X_t, \L_{X_t})\d t +\si_t(X_t)\d W_t,
\end{equation}
where $W_t$ is the $d$-dimensional Brownian motion on a complete filtration probability space $(\OO,\F,\P;\{\F_t\}_{t\ge 0})$, $\L_{X_t}$ is the law of $X_t$ under $\P$, and
$$
B,b: \R_+\times\R^d\times \scr P\to \R^d,\ \ \si: \R_+\times\R^d\to \R^d\otimes\R^d
$$
are measurable.
We will consider the SDE \eqref{E1} with initial distributions in the class
$$
\scr P_2 := \big\{\mu\in \scr P: \mu(|\cdot|^2)<\infty\big\}.
$$
It is well known that $\scr P_2$ is a Polish space under the Wasserstein distance
$$
\W_2(\mu,\nu):= \inf_{\pi\in \C(\mu,\nu)} \bigg(\int_{\R^d\times\R^d} |x-y|^2 \pi(\d x,\d y)\bigg)^{\ff 1 {2}},\ \ \mu,\nu\in \scr P_{2},
$$
where $\C(\mu,\nu)$ is the set of all couplings of $\mu$ and $\nu$. In the following we will assume that $B$ is regular and $b$ is singular in the space variable.

We call \eqref{E1}  strong (resp. weak)  well-posed  for distributions in $\scr P_2,$ i.e. for any initial value $X_0\in L^2(\OO\to\R^d,\F_0;\P)$ (resp. initial distribution $\mu\in \scr P_2$), if
\eqref{E1} has a unique strong (resp. weak) solution with $X_{\cdot}\in C([0,\infty);\scr P_2).$ When $\eqref{E1}$ is both strong and weak well-posed (note that unlike in the classical setting, the strong well-posedness does not imply the weak one), we call it well-posed.
In this case, for any $\mu\in \scr P_2$, denote $(P_t^*\mu=\L_{X_t})_{t\ge 0}$ for the solution $(X_t)_{t\ge 0}$ with initial distribution $\L_{X_0}=\mu\in \scr P_2$. For any $f\in \B_b(\R^d)$, the class of bounded measurable functions on $\R^d$,
we aim to establish Bismut formulas for $P_tf(\mu)$ in $\mu\in \scr P_2$, where
$$P_tf(\mu):= (P_t^*\mu)(f):=\int_{\R^d}f(y) (P_t^*\mu)(\d y),\ \ t>0.$$
To this end, we first recall the  intrinsic/Lions derivatives  for real functions on $\scr P_2$.

\beg{defn} Let   $f: \scr P_2\to \R$.\beg{enumerate}
 \item[$(1)$]  If for any  $\phi\in L^2(\R^d\to\R^d;\mu)$,
  $$D_\phi^I f(\mu):= \lim_{\vv\downarrow 0} \ff{f(\mu\circ({\rm Id}+\vv \phi)^{-1})-f(\mu)}\vv\in\R$$ exists, and is a bounded linear functional in  $\phi$, we call   $f$ intrinsic differentiable  at   $\mu$. In this case, there exists a unique   $D^If(\mu)\in L^2(\R^d\to\R^d;\mu)$  such that
  $$\<D^If(\mu), \phi\>_{L^2(\mu)} = D^I_\phi f(\mu),\ \ \phi\in L^2(\R^d\to\R^d;\mu).$$
We call   $D^If(\mu)$ the  intrinsic derivative  of  $f$ at  $\mu$. If   $f$  is intrinsic differentiable at all $\mu\in \scr P_2$, we call it intrinsic differentiable on  $\scr P_2$ and denote
$$\|D^If(\mu)\|:= \|D^If(\mu)\|_{L^2(\mu)}=\bigg(\int_{\R^d} |D^If(\mu)|^2\d\mu\bigg)^{\ff 1 2}.$$
\item[$(2)$]     If    $f$ is intrinsic differentiable and  for any  $\mu\in\scr P_2$,
  $$\lim_{\|\phi\|_{L^2(\mu)}\to 0} \ff{f(\mu\circ({\rm Id}+\phi)^{-1})-f(\mu)- D^I_\phi f(\mu)}{\|\phi\|_{L^2(\mu)}}=0,$$
we call   $f$  $L$-differentiable on   $\scr P_2$.  In this case,   $D^If(\mu)$ is also denoted by   $D^Lf(\mu)$, and is called the   $L$-derivative of $f$ at $\mu$. \end{enumerate}
\end{defn}

Then intrinsic derivative was first introduced in \cite{AKR} on the configuration space over a Riemannian manifold, while the $L$-derivative appeared in the Lecture notes \cite{Car} for the study of mean field games and is also called Lions derivative in references.

Note that    the derivative  $D^I f(\mu)\in L^2(\R^d\to\R^d;\mu)$  is   $\mu$-a.e.  defined. In applications, we take its continuous version  if exists.
 The following classes of $L$-differentiable functions are often used in analysis: \beg{enumerate}
 \item[{\bf (a)}]   $f\in C^{1}(\scr P_2):$    if  $f$ is   $L$-differentiable such that for every  $\mu\in\scr P_2$,
there exists a   $\mu$-version   $D^L f(\mu)(\cdot)$ such that    $D^L f(\mu)(x)$ is jointly continuous in    $(x,\mu)\in\R^d\times \scr P_2$.
 \item[{\bf (b)}]    $f\in C_b^1(\scr P_2):$   if   $f\in C^1(\scr P_2)$ and
  $D^L f(\mu)(x)$   is bounded.
 \item[{\bf (c)}]   $f\in C^{1,1}(\R^d\times\scr P_2):$  if $f$ is a continuous function on $\R^d\times\scr P_2$ such that
    $f(\cdot,\mu)\in C^{1}(\R^d), f(x,\cdot)\in C^1(\scr P_2)$ with $\nn f(\cdot,\mu)(x)$  and $D^Lf(x,\cdot)(\mu)(y)$ jointly continuous
    in $(x,y,\mu)\in \R^d\times \R^d\times\scr P_2$. If moreover these derivatives are bounded, we denote $f\in C_b^{1,1}(\R^d\times\scr P_2)$.
     \end{enumerate}
We will state the main result in Section 2 and prove it in Section 3.

\section{The main result}
We will assume that $b_t(\cdot,\mu)$ is Dini continuous for which  we introduce the following class as in  \cite{W16}:
 \beg{equation*}\beg{split}
\D= \Big\{\varphi: [0,+\infty)\to [0,+\infty)| \varphi^{2} \text{ is concave and }\varphi \text{ is increasing with } \int_0^1{\frac{\varphi(s)}{s}\d s}<\infty\Big\}.
\end{split}\end{equation*}
 The condition $\int_0^1{\frac{\varphi(s)}{s}\d s}<\infty$ is known as the Dini condition. Clearly, for any $\alpha\in(0,\frac{1}{2})$ the function $\varphi_1(s)=s^{\alpha}$ is in $\D$. Let $\varphi_2(s):=\frac{1}{\log^{1+\delta}(c+s^{-1})}$ for constants $\delta >0$ and $c>0$ large enough such that $\varphi^{2}_2$ is concave, then $\varphi_2$ is also in $\D$. For any (real or $\R^d$-valued) function $f$ on $\R^d$, let
 $$[f]_\varphi:= \sup_{x\ne y\in \R^d} \bigg\{|f(x)|+\ff{|f(x)-f(y)|}{\varphi(|x-y|)}\bigg\},\ \ \varphi\in \D.$$
For a function $f:[0,\infty)\times E\to \R,$ where $E$ is an abstract space, we denote
$$\|f\|_{T,\infty}:=\sup_{t\in [0,T]\times E} |f_t(x)|,\ \ T>0.$$
Throughout this paper, we make the following assumption.
\beg{enumerate}
\item[$\bf{(H)}$]
For each $t\geq0$ and $x\in\R^d$,
$b_t(x,\cdot)\in C^{1}(\scr P_2)$ with $D^L b_t(x,\mu)(y)$    continuous in $(x,y,\mu)\in \R^d\times\R^d\times\scr P_2$,
$B_t\in C^{1,1}(\R^d\times\scr P_2)$, $\sigma_t\in C^1(\R^d)$ is invertible, such that  for any $T>0$,
$$   \big\|(\|\sigma\| +\|\si^{-1}\|  + |B(0,\dd_0)|+
 [b]_\varphi +\|\nabla\sigma\| +\|D^L b\| +\|\nn B\| + \|D^LB\| )\big\|_{T,\infty} <\infty
$$ holds for some $\varphi\in \D$, where $\dd_0$ is the Dirac measure at $0\in\R^d$, $[\cdot]_{\varphi}$ is the modulus of continuity in $x\in\R^d$, and $D^L$ is Lion's derivative in $\mu\in \scr P_2$.
\end{enumerate}
 The following is a simple example for $b$ satisfying {\bf(H)}.
\begin{exa}
Let $h(x)= |x|^{\alpha}, x\in\mathbb{R}^d$ for $\alpha\in(0,\frac{1}{2})$, let
$f=(f_1,\cdots,f_m)\in C_b^1(\R^d;\R^m)$ for some $m\ge 1$, and let
 $$F: [0,T]\times\R\times\R^m\to \R^d$$ be measurable and bounded such that
 $$c:= \sup_{r\in\R, z\in\R^m, t\in [0,T]} \big\{  |\pp_r F_t(r,z)|+|\nn_z F_t(r,z)|\big\}<\infty.$$
Then
$$b_t(x):= F_t(h(x), \mu(f)),\ \ \mu(f):= \int_{\R^d} f\d\mu$$
satisfies {\bf (H)}. Indeed, we have
$$|b_t(x,\mu)- b_t(y,\mu)|\le c |h(x)-h(y)|\le c|x-y|^\aa,$$
and
$$D^L b_t(x,\mu)= \sum_{i=1}^m \big\<\pp_{z_i} F_t(h(x),z)|_{z=\mu(f)},\nn f_i\big\>$$
satisfies
$$\sup_{t\in [0,T], x\in\R^d, \mu\in \scr P_2} \|D^L b_t(x,\mu)\|\le c \sum_{i=1}^m \|\nn f_i\|_\infty<\infty.$$

\end{exa}

 \begin{rem}\label{ES1}
With the second inequality in \eqref{N*} replacing {\rm \cite[(27)]{HW20a}},  and   with $t$ replacing  $A_t$ in {\rm \cite[(35)]{HW20a}},      the proof of
{\rm \cite[Theorem 1.1(2)]{HW20a}} yields that
{\bf (H)} implies the well-posedness of \eqref{E1}   for distributions in $\scr P_2$.
 We will show that  this assumption also ensures the intrinsic differentiability of  $P_Tf$   for $T>0$ and $f\in \B_b(\R^d)$.  To prove the $L$-differentiability of $P_Tf$, we make the following additional assumption. \end{rem}

\beg{enumerate}\item[{\bf (C)}] For any $T>0$, there exists $p>2$  such that
$$\sup_{(t,x,\mu)\in [0,T]\times\R^d\times\scr P_2} \int_{\R^d}  |D^L(b+B)_t(x,\mu)(y)|^p \mu(\d y)<\infty.$$
\end{enumerate}
Obviously, {\bf (C)} holds if $D^L(b_t+B_t)(x,\mu)(y)$ is bounded in $(t,x,y,\mu)\in [0,T]\times \R^d\times\R^d\times\scr P_2.$

For any $T>0$, let $\C_T:=C([0,T];\R^d)$ be equipped with the uniform norm. For $\vv\in(0,1)$ and $\eta,X_0\in L^2(\Omega\to\R^d,\scr F_0;\P)$,
let $\{X^{\eta,\vv}_t\}_{t\ge 0}$ solve \eqref{E1} with initial value $X_0+\vv\eta$.
The main result of the paper is the following.

\beg{thm}\label{T3.1}
Assume {\bf (H)}.   Then the following statements hold.
\beg{enumerate}
\item[$(1)$] For any $T>0$, the limit
\begin{align}\label{DX}
\nn_\eta X_t:=\lim\limits_{\vv\rightarrow0}\frac{1}{\vv}(X^{\eta,\vv}_t-X_t)
\end{align}
exists in $L^2(\Omega\to\C_T;\P)$, and   there exists a constant $C_T>0$   such that
\begin{align}\label{GE}
\E\left(\sup\limits_{t\in[0,T]}|\nn_\eta X_t|^2\right)
\leq C_T\E|\eta|^2,\ \ \eta,X_0\in L^2(\Omega\to\R^d,\scr F_0;\P).
\end{align}
\item[$(2)$] $P_Tf$ is intrinsically differentiable for any $T>0$ and $f\in \B_b(\R^d)$, and
\beq\label{LH2}
D_{\phi}^I(P_Tf)(\mu)=\mathbb{E}\left(f(X_T)\int_0^T\<\zeta_t^\phi,\d W_t\>\right),\ \ \mu\in\scr P_2, \phi\in L^2(\R^d\to\R^d;\mu)
\end{equation}
holds for  $X_t$ solving  $\eqref{E1}$ with $\L_{X_0}=\mu$, and
$$
\zeta_t^\phi:=\sigma_t (X_t)^{-1}\Big\{g_t'\nn_{\phi (X_0)} X_t
+g_t\E\big[D^L(B_t+b_t)(y,  \L_{X_t})(X_t)\nn_{\phi(X_0)} X_t\big]
\big|_{y=X_t}\Big\}
$$ for $t\in [0,T]$ and $g\in C_b^1([0,T])$ with $g_0=0, g_T=1.$  Consequently,    there exists an increasing function   $C: [0,\infty)\to (0,\infty)$ such that
\beq\label{A1}
\|(D^I_\phi P_tf)(\mu)\|
\leq \ff{C_t}{\ss t} \big\{P_tf^2(\mu)-(P_tf(\mu))^2\big\}^{\ff 1 2},\ \  t>0.\end{equation}
\item[$(3)$]  Assume further   {\bf (C)}.   Then   $P_Tf$ is $L$-differentiable for any  $T>0$ and  $f\in \B_b(\R^d)$. As a result, for any $t>0$ and $\mu,\nu\in \scr P_2$,
 \beq\label{A2} \|P_t^*\mu-P_t^*\nu\|_{var}:= \sup_{\|f\|_\infty\le 1} |(P_t^*\mu)(f)- (P_t^*\nu)(f)|\le \ff{C_t}{\ss t} \W_2(\mu,\nu).\end{equation}
\end{enumerate} \end{thm}

\paragraph{Remark 2.1.} When $b=0$, the Bismut formula and the $L$-differentiability of $P_Tf$ for any $f\in \B_b(\R^d)$ have been proved in \cite{RW}. This is now included in  Theorem \ref{T3.1} as a special case.  When $b\ne 0$, to manage this singular term we have to use Zvonkin's transforms depending on the parameter in initial distributions.

\paragraph{Remark 2.2.} For fixed $\mu\in \scr P_2$, consider the following decoupled SDE:
\begin{align*}
\d X_t^{x,\mu}= b_t(X^{x,\mu}_t,P_t^*\mu)\d t+ \sigma_t(X^{x,\mu}_t,P_t^*\mu)\d W_t,  \ \ \ X_0^{x,\mu}=x\in\mathbb R^n.
\end{align*} Let $p_t^\mu(x,y)$ be the distribution density function of $X_t^{x,\mu}$.
Derivatives of $p_t^\mu(x,y)$ in both $x$ and $\mu$ have  been presented in  \cite{CN}, where $b$ and $\sigma$ are assumed to be $\eta$-H\"older continuous for some $\eta\in(0,1]$ with respect to the spatial variable. In particular, these estimates imply   estimates on $D^L P_tf(\mu)$ for $f\in \B_b(\R^d)$. In fact, let   $P_t^\mu$ be the  transition semigroup of $X_t^{x,\mu}$. Then we have
$$
P_tf(\mu)=\int_{\mathbb R^n}P_t^\mu f(x)\mu(\d x)=\int_{\mathbb R^n\times\mathbb R^n}f(y)p_t^\mu(x,y)\d y\mu(\d x).
$$
Consequently,
\begin{align*}
D^LP_tf(\mu)(z)=\int_{\mathbb R^n\times\mathbb R^n}f(y)\left\{D^Lp_t^\mu(x,y)(z)\right\}\d y\mu(\d x)+\int_{\mathbb R^n}f(y)\nabla_zp_t^\mu(z,y)\d y.
\end{align*}
This combined with \cite[(3.9),Theorem 3.6]{CN} yields
\begin{align*}
\|D^LP_tf(\mu)(\cdot)\|_\infty\leq C\|f\|_\infty\left(t^{-\frac{1}{2}}\vee t^{-\frac{1-\eta}{2}}\right).
\end{align*}

\section{Proof of Theorem \ref{T3.1} }
By {\bf (H)},  \eqref{LH2} implies that  $P_Tf$ is intrinsically differentiable for $T>0$ and $f\in \B_b(\R^d)$,  and \eqref{A1} holds for some increasing function $C$. Then
 \eqref{A2} follows from  the estimate (see \cite{Car}):
\beq\label{LPS} |f(\mu)-f(\nu)|\le \sup_{\gg\in \scr P_2}\|D^L f(\gg)\| \W_2(\mu,\nu),\ \ \mu,\nu\in \scr P_2.\end{equation}
  To prove \eqref{LPS}, let $X,Y$ be two random variables such that
$$\L_{X}=\mu,\ \ \L_Y=\nu,\ \ \E|X-Y|^2=\W_2(\mu,\nu)^2.$$
Then by taking $X_s= (1-s)X+sY$ for $s\in [0,1],$ \eqref{LPS} follows from
 the following chain rule for distributions of random variables, which is taken from \cite[Theorem 2.1]{BRW}, see also
  \cite[Theorem 6.5]{Car} and  \cite[Proposition 3.1]{RW} for earlier results under stronger conditions.

\begin{lem}\label{C-R} Let $\{X_\vv\}_{\vv\in [0,1]}$ be a family of random variables on $\R^d$ such that $\dot X_0:= \lim_{\vv\downarrow 0} \ff 1\vv(X_\vv-X_0)$ exists in $L^2(\OO\to\R^d;\P)$. Let $f$ be a real function on $\scr P_2$. If either $\mu:=\L_{X_0}$ is atomless and $f$ is $L$-differentiable at $\mu$, or $f\in C^1$ in a neighborhood $U$ of $\mu$ such that
$$|D^Lf(\mu)(x)|\le c(1+|x|),\ \ x\in\R^d, \mu\in U$$ holds for some constant $c>0$, then
$$\lim_{\vv\downarrow 0} \ff {f(\L_{X_\vv})-f(\mu)}\vv=\E\<D^Lf(\mu)(X_0), \dot X_0\>.$$
 \end{lem}

Therefore, it suffices to prove Theorem \ref{T3.1}(1), the formula \eqref{LH2}, and Theorem \ref{T3.1}(3).

\subsection{Proof of Theorem \ref{T3.1}(1)}
We first explain that we may   assume
 \beq\label{BBB}
B_t(\cdot,\mu)\in C^3(\R^d),\ \  \sum_{j=1}^3\|\nn^j B\|_{T,\infty}<\infty,\ \ T>0.\end{equation}
Indeed, for $0\le \rr\in C_0^\infty(\R^d)$ with $\int_{\R^d}\rr(x)\d x=1,$ let
$$\tt B_t(x,\mu):= \int_{\R^d}B_t(y,\mu)\rr(x-y)\d y,\ \ t\ge 0, x\in \R^d,\mu\in \scr P_2.$$
By {\bf (H)}, this implies \eqref{BBB} for $\tt B$ replacing $B$, and that $B_t-\tt B_t$ is bounded with $\|\nn(B-\tt B)\|_{T,\infty}<\infty$. By combining $B_t-\tt B_t$ with $b$, we may  and do assume that $B$ satisfies \eqref{BBB}.

Next, we make the following distribution dependent Zvonkin's transform to regularize the SDE \eqref{E1}.
For any $\lambda\ge 0,$   $T>0$ and $\hat\mu\in \C_{T,\scr P_2}:=C([0,T];\scr P_2)$, let
$$b_t^{\hat\mu}(x):= b_t(x,\hat \mu_t),\ \ B_t^{\hat\mu}(x):= B_t(x,\hat \mu_t)\ \ t\in [0,T], x\in\R^d,$$ and consider the following PDE for $u^{\ll,\hat\mu}: [0,T]\times\R^d\to \R^d$:
\beq\label{PDE}
 \partial_t u_t^{\ll,\hat\mu} +\frac{1}{2}\mathrm{Tr} (\sigma_t\sigma_t^\ast\nabla^2u_t^{\ll,\hat\mu})+\nabla_{b_t^{\hat\mu}+B_t^{\hat \mu}}u_t^{\ll,\hat\mu}+b_t^{\hat\mu}=\lambda u_t^{\ll,\hat\mu},\ \ u_T^{\ll,\hat\mu}=0.
\end{equation}
To solve this equation, we consider the flow induced by $B_t^{\hat\mu}$:
$$\pp_t \psi_t=  B_t^{\hat\mu}\circ \psi_t,\ \ \psi_{T}(x)=x,\ \ t\in [0,T], x\in\R^d.$$
By \eqref{BBB},  $\psi_t$ is a diffeomorphism on $\R^d$ and
 \beq\label{PSS} \sup_{t\in [0,T],\hat\mu\in \C_{T,\scr P_2}} \sum_{j=1}^3\big\{\|\nn^j \psi_t\|_\infty+\|\nn^j\psi_t^{-1}\|_\infty\big\}<\infty.\end{equation}
Let
\begin{align*}
&\bar{\sigma}_t=\sqrt{(\nabla\psi_t)^{-1}(\sigma_t\sigma_t^\ast)(\psi_t)[(\nabla\psi_t)^{-1}]^\ast},\\
&\bar{b}_t^{\hat\mu}=(\nabla\psi_t)^{-1}b_t^{\hat\mu}(\psi_t)
-\frac{1}{2}\sum_{j=1}^d\mathrm{Tr}[(\sigma_t\sigma^\ast_t)(\psi_t) [(\nabla\psi_t)^{-1}]^\ast\nabla^2\psi_t^j(\nabla\psi_t)^{-1}](\nabla\psi_t)_j^{-1}.
\end{align*}
By \cite{W16},   {\bf (H)} and \eqref{PSS}  imply that the PDE
\beq\label{PDE9}
\pp_t \bar u_t^{\ll,\hat\mu} +\frac{1}{2}\mathrm{Tr} (\bar{\sigma}_t\bar{\sigma}_t^\ast\nabla^2\bar u_t^{\ll,\hat\mu})+\nabla_{\bar{b}_t^{\hat\mu}}\bar{u}_t^{\ll,\hat\mu}+b_t^{\hat\mu}\circ \psi_t=\lambda \bar u_t^{\ll,\hat\mu},\ \ \bar{u}_T^{\ll,\hat\mu}=0
\end{equation}
has a unique solution with
$$\lim_{\ll\to \infty} \sup_{\hat\mu\in \C_{T,\scr P_2}}\|\nn \bar{u}^{\ll,\hat\mu}\|_{T,\infty}=0,\ \ \sup_{\ll\ge 0,\hat\mu\in \C_{T,\scr P_2}} \|\nn^2 \bar{u}^{\ll,\hat\mu}\|_{T,\infty}<\infty.$$
So, {$u_t^{\ll,\hat\mu}:= \bar u_t^{\ll,\hat \mu}\circ \psi_t^{-1}$ } solves \eqref{PDE} with
\beq\label{N*} \lim_{\ll\to \infty} \sup_{\hat\mu\in \C_{T,\scr P_2}}\|\nn u^{\ll,\hat\mu}\|_{T,\infty}=0,\ \ \sup_{\ll\ge 0,\hat\mu\in \C_{T,\scr P_2}} \|\nn^2 u^{\ll,\hat\mu}\|_{T,\infty}<\infty.\end{equation} By the uniqueness of \eqref{PDE9} and that a solution $u^{\ll,\hat\mu}_t$ to \eqref{PDE} also gives a solution $\bar u_t^{\ll,\hat \mu}:= u_t^{\ll,\hat \mu}\circ\psi_t$  to \eqref{PDE9}, \eqref{PDE} has a unique solution.
By \eqref{PSS} and  \eqref{N*},  there exists a universal constant $\ll_0>0$ such that
\beq\label{u0}
\sup_{\hat\mu\in\C_{T,\scr P_2}}\|\nabla u^{\lambda,\hat\mu}\|_{T,\infty}\leq
\frac{1}{5},\ \
\sup_{\hat{\mu}\in \C_{T,\scr P_2}}\|\nabla^2 u^{\lambda,\hat\mu}\|_{T,\infty}\le \ll_0,\ \ \ll\ge \ll_0.
\end{equation}
For any $\mu\in \scr P_2$, let $X_0$ be $\F_0$-measurable
with $\L_{X_0}=\mu$, and simply denote
$$u_t^{\ll,\mu}:= u_t^{\ll,\hat\mu}\ \text{for}\ \hat\mu_t=P_t^*\mu,\ \
b_t^\mu(x):= b_t(x, P_t^*\mu),\ \ B_t^\mu(x):= B_t(x, P_t^*\mu), \ \ (t,x)\in [0,T]\times\R^d.
$$
Let $\theta^{\lambda,\mu}_t(x)=x+u^{\lambda,\mu}_t(x)$.
By \eqref{PDE} and It\^o's formula, we derive
\beq\label{E-X}
\d \theta^{\lambda,\mu}_t(X_t)
= \big\{B_t^{\mu}(X_t)+\lambda u^{\lambda,\mu}_t(X_t)\big\}\d t+\big\{(\nn\theta_t^{\lambda,\mu})\sigma_t\big\}(X_t)\,\d W_t.
\end{equation}
Then
\beq\label{YT}Y_t:=\theta^{\lambda,\mu}_t(X_t),\ \ t\in [0,T]\end{equation} solves the SDE
\begin{align}\label{E-Y0}
\d Y_t=\tilde{b}^\mu_t(Y_t)\d t+ \tilde{\sigma}^\mu_t(Y_t)\,\d W_t,\ \ Y_0=\theta^{\lambda,\mu}_0(X_0)
\end{align}
where
\begin{align}\label{t-b-sig}
\tilde{b}^\mu_t:=(B_t^\mu+\lambda u^{\lambda,\mu}_t)\circ(\theta^{\lambda,\mu}_t)^{-1},\ \ \tilde{\sigma}^\mu_t:=\big\{(\nn\theta_t^{\lambda,\mu})\sigma_t\big\}\circ(\theta^{\lambda,\mu}_t)^{-1},\ \ t\in[0,T].
\end{align}
By {\bf (H)} and \eqref{u0},  we have
\begin{align}\label{g-tb}
\sup_{\mu\in \scr P_2}\big\{\|\nabla\tilde{b}^\mu\|_{T,\infty}+ \|\nn\tt\si^\mu\|_{T,\infty}\big\}
<\infty.
\end{align}
Let $\eta, X_0\in L^2(\OO\to\R^d,\F_0;\P)$. For any $\vv\ge 0$, let
 $X_t^\vv$ solve  the SDE
\begin{align}\label{eps}
\d X_t^\vv=(B_t+b_t)(X_t^\vv,\L_{X_t^\vv})\d t +\si_t(X_t^\vv) \d W_t,
\ \ X_0^\vv=X_0+\vv\eta.
\end{align}
By It\^o's formula and \eqref{t-b-sig},  $Y_t^\vv:= \theta_t^{\ll,\mu}(X_t^\vv)$ solves the SDE
\beq
\begin{split}\label{E-Y}
\d Y_t^\vv=&\,\Big\{\tilde{b}^\mu_t(Y_t^\vv)+ \big[\nn\theta_t^{\lambda,\mu}[(B_t+b_t)(\cdot,\L_{\cdot})-(B^\mu_t+b^\mu_t)]\big]\big((\theta^{\lambda,\mu}_t)^{-1}(Y_t^\vv)\big)\Big\}\d t\\
&+\tilde{\sigma}^\mu_t(Y_t^\vv)\d W_t,\ \ t\in [0,T], Y^\vv_0=\theta^{\lambda,\mu}_0(X_0+\vv\eta).
\end{split}
\end{equation}

\begin{lem}\label{Le-xi}
Under $\bf{(H)}$, the family $\{\xi_t^\vv:= \vv^{-1}(Y_t^\vv-Y_t)\}_{t\in [0,T], \vv\in (0,1]}$   is $L^2$-uniformly integrable, i.e.
\begin{align}\label{n2}
\lim_{n\to\infty}\sup\limits_{\vv\in(0,1]}\E\left(\sup\limits_{t\in[0,T]}|\xi^\vv_t|^21_{\{\sup\limits_{t\in[0,T]}|\xi^\vv_t|^2\geq n\}}\right)=0.
\end{align}
\end{lem}
\begin{proof} By \eqref{LPS}, \eqref{t-b-sig}, \eqref{g-tb} and \eqref{E-Y}, we find a constant $c>0$ such that for any $\vv\in (0,1]$, the It\^o's formula gives
\beq\label{*A1} \d |\xi_t^\vv|^2\le c\big(|\xi_t^\vv|^2+\E|\xi_t^\vv|^2\big)\d t+ c |\xi_t^\vv|^2 \d M_t^\vv,\ \ t\in [0,T],|\xi_0^\vv|^2\le c|\eta|^2,\end{equation}
for some martingale $M_t^\vv$ with $\d\<M^\vv\>_t\le \d t.$ By BDG's inequality,   this implies
\begin{align}\label{p2}\E \Big[\sup_{t\in [0,T]} |\xi_t^\vv|^2\Big]\le c_1\E|\eta|^2,\ \ \vv\in [0,1]
 \end{align}for some constant $c_1>0.$
Combining this with \eqref{*A1} we obtain
 $$\d \big\{|\xi_t^\vv|^2 \e^{-c(t+M_t^\vv)}\big\}\le cc_1(\E|\eta|^2) \e^{-c(t+M_t^\vv)}\d t,$$
 so that for $$N_\vv:= \frac{1}{2}\sup_{s\in [0,T]}\e^{2cM_s^\vv}+\frac{1}{2}\sup_{s\in [0,T]} \e^{-2cM_s^\varepsilon},$$ we find a constant $c_2>0$ such that
\beq\label{xi-vv}\begin{split}D_\vv:&= \sup_{t\in [0,T]} |\xi_t^\vv|^2 \le \e^{cT}|\eta|^2 N_\vv +cc_1(\E|\eta|^2) T\e^{cT}N_\vv\\
&\le c_2(|\eta|^2 +\E|\eta|^2) N_\vv,\   \vv \in (0,1].\end{split}\end{equation}

For any $n>m>1$, by BDG's inequality and $\d\<M^\vv\>_t\le \d t$, we find a constant $K>0$ such that
this implies
\beg{align*} &\E(D_\vv- c_2n)^+ \le c_2 \E\big[((|\eta|^2 +\E|\eta|^2)N_\vv- n)^+\big] \\
&\le c_2 m \E\big[1_{\{|\eta|^2+\E|\eta|^2\le m\}}(N_\vv-n/m)^+\big] + c_2\E\big[(|\eta|^2+\E|\eta|^2)1_{\{|\eta|^2+\E|\eta|^2>m\}}\E(N_\vv|\F_0)\big]\\
&\le \ff{c_2m}{n/m}  \E(N_\vv^2)+ K   \E\big[(|\eta|^2+\E|\eta|^2)1_{\{|\eta|^2+\E|\eta|^2>m\}}\big],\ \ \vv\in (0,1].\end{align*}
Taking supremum with respect to $\vv\in(0,1]$ on both sides, and letting first $n\to\infty$ then $m\to \infty$ we finish the proof.
\end{proof}

For ($L$-)differentiable (real, vector, or matrix valued) functions $f$ on $\R^d$ and   $g$ on $\scr P_2$, let
\begin{align}\label{t-XI}
\tilde{\Xi}^\vv_g(t):=\frac{1}{\vv}(g(\L_{X^\vv_t})-g(\L_{X_t}))
-\E\<D^Lg(\L_{X_t})(X_t),\nn(\theta_t^{\lambda,\mu})^{-1}(Y_t)\xi^\vv_t\>, \end{align}
\begin{align}\label{XI}
\Xi^\vv_f(t):=\frac{1}{\vv}(f(Y_t^\vv)-f(Y_t))-\nn_{\xi^\vv_t}f(Y_t),\ \ \vv>0, t\in [0,T].
\end{align}
The following lemma can be proved by using Lemma \ref{C-R}, \eqref{Le-xi}  and the argument in the proof of \cite[Lemma 3.4]{RW},   we omit the details to save space.

\begin{lem}\label{Le-Xi}
Assume $\bf{(H)}$. For any function $f\in C^1(\R^d)$ and $g\in C_b^1(\scr P_2)$ with $ \|\nn f\|_\infty+\sup\limits_{\mu\in\scr P_2}
\|D^Lg(\mu)\|<\infty,$
there exists a constant $C>0$ such that
$$
|\Xi^\vv_f(t)|^2\leq C\|\nn f\|^2_\infty|\xi^\vv_t|^2, \ \ |\tilde{\Xi}^\vv_g(t)|^2\leq C\|D^L g\|_\infty^2\E|\xi^\vv_t|^2, \ \ t\in[0,T],
$$
and
$$
\lim\limits_{\vv\rightarrow0}(\E|\Xi^\vv_f(t)|^2+|\tilde{\Xi}^\vv_g(t)|^2)=0,\ \ t\in[0,T].
$$
\end{lem}

\begin{lem}\label{ti-v}
Assume $(\bf{H})$. Then the limit
\begin{align}\label{v-phi}
\nn_\eta Y_t:=\lim_{\vv\to 0}\frac{1}{\varepsilon}(Y_t^\vv-Y_t)
\end{align}
exists in $L^2(\Omega\to\C_T;\P),$ and is the unique solution to the   linear equation
\begin{equation}\label{v-phi-eq}\beg{split}
v_t^\eta
=&\,[\nn\theta^{\lambda,\mu}_0](X_0)\eta
+\int_0^t\Big\{\nn_{v_s^\eta}\tilde{b}^\mu_s(Y_s)+ F_s^\mu\big(X_s\big)v_s^\eta \Big\}\d s\\
&+\int_0^t\nn_{v_s^\eta}\tilde{\sigma}^\mu_s(Y_s)\,\d W_s, \ \ t\in [0,T],\end{split}\end{equation} where for any random variables $X,v$ on $\R^d$ and $s\in [0,T]$,
\beq\label{FSM}F_s^\mu(X)v:= \{\nn\theta_s^{\lambda,\mu}\}(X)
\E\big[D^L(B_s+b_s)(y,\L_{X})(X)(\nn \theta^{\lambda,\mu}_s)^{-1} (X)v\big]\big|_{y=X}.\end{equation}
Consequently, for any $p\ge 1$ there exists   a constant $c>0$ such that
\begin{align}\label{ve}
\E\Big(\sup\limits_{t\in[0,T]}|\nn_\eta Y_t|^p\Big|\F_0\Big)
\leq c\big(|\eta|^p + (\E|\eta|^2)^{\ff p2}\big),\ \ \eta,X_0\in L^2(\OO\to\R^d,\F_0;\P).
\end{align}
\end{lem}
\begin{proof}
For the existence of $\nabla_\eta Y_t$   in $L^2(\Omega\to\C_T;\P)$,
it suffices to verify
\begin{align}\label{xi}
\lim\limits_{\vv,\delta\rightarrow0}
\E\Big(\sup\limits_{t\in[0,T]}|\xi^\vv_t-\xi^\delta_t|^2\Big)=0.
\end{align}
By \eqref{FSM}, (\ref{E-Y}) and (\ref{XI}), we obtain
\begin{equation}\label{x-e}
\xi^\vv_t=
 \frac{1}{\vv}(\theta^{\lambda,\mu}_0(X_0+\vv\eta)-\theta^{\lambda,\mu}_0(X_0))  +\int_0^t A_s^\vv\d s+ \int_0^t B_s^\vv\d W_s,\end{equation}
 where for $(\Xi,\tt\Xi)$ in \eqref{t-XI} and \eqref{XI},
 \begin{align*} &A_s^\vv:=  \frac{1}{\vv}\Big\{(\tilde{b}^\mu_s(Y_s^\vv)-\tilde{b}^\mu_s(Y_s))+[\nn\theta_s^{\lambda,\mu}[(B_s+b_s)(\cdot,\L_{\cdot})-B_s^\mu-b^\mu_s)]]
((\theta^{\lambda,\mu}_s)^{-1}(Y_s^\vv))\Big\}\\
&=\Xi^\vv_{\tilde{b}_s^\mu}(s)+\nn_{\xi^\vv_s}\tilde{b}^\mu_s(Y_s)+ F_s^\mu(X_s^\vv)\xi_s^\vv
  +\{\nn\theta_s^{\lambda,\mu}(X^\vv_s)\} \tt\Xi^\vv_{(B_s+b_s)(X_s^\vv,\cdot)}(s),\end{align*}
$$B_s^\vv:= \frac{1}{\vv}(\tilde{\sigma}^\mu_s(Y_s^\vv)-\tilde{\sigma}^\mu_s(Y_s))=\Xi^\vv_{\tilde{\sigma}^\mu_s}(s)+\nn_{\xi^\vv_s}\tilde{\sigma}^\mu_s(Y_s).
$$
Since  $\|\nn^2\theta_0^{\ll,\mu}\|_\infty<\infty$ due to \eqref{u0}, we have
$$\lim_{\vv,\dd\downarrow 0} \E\bigg|\frac{\theta^{\lambda,\mu}_0(X_0+\vv\eta)-\theta^{\lambda,\mu}_0(X_0)}{\vv} -\ff {\theta^{\lambda,\mu}_0(X_0+\dd\eta)-\theta^{\lambda,\mu}_0(X_0)} \dd\Big|^2=0.$$
Moreover, by Lemma \ref{Le-Xi}, {\bf (H)} and \eqref{u0} imply
\begin{align*}
&\lim_{\vv,\dd\to 0} \E\big\{|\tt{\Xi}^\vv_{(B_s+b_s)(X_s^\vv,\cdot)}(s)|^2+ |\Xi^\vv_{\tt \si^\mu_s}(s)|^2+|\Xi_{\tt b_s^\mu}^\vv(s)|^2\\
&\qquad\qquad\qquad+\tt{\Xi}^\dd_{(B_s+b_s)(X_s^\dd,\cdot)}(s)|^2+|\Xi^\dd_{\tt \si^\mu_s}(s)|^2+|\Xi_{\tt b_s^\mu}^\dd(s)|^2+|F_s^\mu(X_s^\vv)-F_s^\mu(X_s^\dd)|^2 \big\} =0.
\end{align*}
Combining these with \eqref{x-e} and applying BDG's inequality,  we find a constant $c>0$ such that
$$\limsup_{\vv,\dd\downarrow 0} \E \Big[\sup_{s\in [0,t]} |\xi_s^\vv-\xi_s^\dd|^2 \Big]\le c \int_0^t \limsup_{\vv,\dd\downarrow 0} \E
\Big[\sup_{s\in [0,r]} |\xi_s^\vv-\xi_s^\dd|^2\Big] \d r,\ \ t\in [0,T].$$
Since  Lemma \ref{Le-xi} ensures $$\limsup_{\vv,\dd\downarrow 0} \E \sup_{s\in [0,T]}\Big[ |\xi_s^\vv-\xi_s^\dd|^2 \Big]<\infty,$$
by Gronwall's lemma we prove  \eqref{xi}.

Finally, by {\bf (H)}, Lemma \ref{Le-Xi} and the definitions of $F_s^\mu, \nn_\eta Y_s, A_s^\vv,B_s^\vv$ and the fact
$X_s=(\theta_s^{\ll,\mu})^{-1}(Y_s),$ we may let $\vv\downarrow 0$ in \eqref{x-e} to derive \eqref{v-phi-eq}, which together with {\bf (H)} implies
$$\E|v_t^\eta|^2\le c_1 \E |\eta|^2,\ \ t\in [0,T].$$
Combining this with \eqref{v-phi-eq}, {\bf (H)} and BDG's inequality, we prove
\eqref{ve} for some constant $c>0$.
\end{proof}

\beg{proof}[Proof of Theorem $\ref{T3.1}(1)$] By $(\ref{v-phi})$, Lemma \ref{Le-Xi}  and
\eqref{u0}, we can find a constant $c>0$ such that
\begin{align*}
&\lim_{\vv\to0}\E\left( \sup_{t\in[0,T]}\left|\frac{1}{\vv}(X_t^\vv-X_t)-\{\nn(\theta_t^{\lambda,\mu})^{-1}\}
(Y_t)\nn_\eta Y_t\right|^2\right)\\
=&\lim_{\vv\to0}\E\left(\sup_{t\in[0,T]}\left|\frac{1}{\vv}
((\theta_t^{\lambda,\mu})^{-1} (Y_t^\vv)-(\theta_t^{\lambda,\mu})^{-1}(Y_t))
-\{\nn(\theta_t^{\lambda,\mu})^{-1}\}(Y_t)\nn_\eta Y_t\right|^2\right)\\
\leq&c\lim_{\vv\to0}\E\sup_{t\in[0,T]} |\xi_t^\varepsilon-\nn_\eta Y_t|^2=0.
\end{align*}
This implies
\beq\label{DXX}
\nn_\eta X_t:=\lim\limits_{\vv\rightarrow0}\frac{1}{\vv}(X_t^\vv-X_t)= \big\{\nn(\theta_t^{\lambda,\mu})^{-1}\big\}
(Y_t)\nn_\eta Y_t
\end{equation}
exists in $L^2(\Omega\to\C_T;\P)$,    together with \eqref{u0} and \eqref{ve}, yields (\ref{GE}).
\end{proof}

\subsection{Proof of \eqref{LH2}}

\begin{lem}\label{BYY}
Assume {\bf(H)}.  Let $k\in [0,T)$ and  $g\in C_b^1([0,T])$ with $g_k=0$ and $g_T=1$.  Then
 \beq\label{LH2YYY}\beg{split}
&\nn_\eta\E (f(Y_T)|\F_k):= \lim_{\vv\to 0}\frac{\E (f(Y^\vv_T)-  f(Y_T)|\F_k)}{\vv}\\
&=\mathbb{E}\bigg(f(Y_T)\int_k^T\<\zeta_t^\eta,\d W_t\>\bigg|\F_k\bigg),\ \ f\in\B_b(\R^d),\eta\in L^2(\OO\to\R^d,\F_0;\P),
\end{split}\end{equation}
where
\beq\label{t-Zeta}
\zeta_t^\eta
:=\sigma_t(X_t)^{-1}  \Big\{g_t' \nn_\eta X_t
+g_t \E[D^L(B_t+b_t)(y,\L_{X_t})(X_t)  \nn_\eta X_t]|_{y=X_t}\Big\}.
\end{equation}
Consequently, it holds
\beq\label{LH2YY}\beg{split}
&\nn_\eta\E (f(X_T)|\F_k):= \lim_{\vv\to 0}\frac{\E (f(X^\vv_T)-  f(X_T)|\F_k)}{\vv}\\
&=\mathbb{E}\bigg(f(X_T)\int_k^T\<\zeta_t^\eta,\d W_t\>\bigg|\F_k\bigg),\ \ \ \ f\in\B_b(\R^d),\eta\in L^2(\OO\to\R^d,\F_0;\P).
\end{split}\end{equation}

\end{lem}

\begin{proof}
Simply denote $v_t=g_t\nn_\eta Y_t$ for $t\in[0,T]$. By It\^o's formula, (\ref{v-phi-eq}) and \eqref{DXX}, we obtain
\beq\label{MD}\beg{split}
\d v_t=&\nn_{v_t}\tilde{b}^\mu_t(Y_t)\d t+\nn_{v_t}\tilde{\sigma}^\mu_t(Y_t)\,\d W_t
+g_t'v_t\d t\\
&+g_t\nn\theta_t^{\lambda,\mu}(X_t)
\E[D^L(B_t+b_t)(y,\L_{X_t})(X_t)\nn_\eta X_t ]|_{y=X_t}\d t,\ \
 v_k=0, t\in[k,T].
\end{split}\end{equation}
On the other hand, let
 $h_t=\int_{k}^t\zeta_s^\eta\d s$ for $t\in [k,T]$.
By (\ref{g-tb})    and \cite[Theorem 2.2.1]{DN}, the Malliavin derivative $D_hY_t$ of $Y_t$ along $h$ satisfies
$$
\d D_h Y_t
=\nn_{D_hY_t}\tilde{b}^\mu_t(Y_t)\d t+\nn_{D_h Y_t}\tilde{\sigma}^\mu_t(Y_t)\,\d W_t
+\tilde{\sigma}^\mu_t(Y_t)\,\d h_t,
\ \  D_h Y_k=0, t\in [k,T].$$
By  the definition of $h$ we see that $D_h Y_t$ solves  \eqref{MD}, so that the uniqueness implies
 $v_t=D_h Y_t$. In particular, $\nn_\eta Y_T=v_T^\eta= D_hY_T$.
Thus, for any $f\in C_b^1(\mathbb{R}^d)$, by   the dominated convergence theorem due to \eqref{ve}, and the integration by parts formula for Malliavin derivative (\cite[Lemma 1.2.1]{DN}), we  obtain
 \begin{align*}
\nn_\eta\E( f(Y_T)|\F_k)
&=\E(\nn_{v_T^\eta} f(Y_T)|\F_k)\\
&=\E(\nn_{D_h Y_T}f(Y_T)|\F_k)\\
&=\E (D_hf(Y_T)|\F_k)
=\E\Big(f(Y_T)\!\!\int_k^T\!\!\<\zeta_t^\eta,\d W_t\>|\F_k\Big).
\end{align*} So, \eqref{LH2YYY} holds for $f\in C_b^1(\R^d).$
By an approximation argument (see \cite[Page 4764]{RW}), the formula also holds for $f\in\B_b(\R^d)$.
Since $\theta_T^{\ll, \mu}=\mathrm{Id}$, we have $(X_T,X_T^\vv)= (Y_T,Y_T^\vv)$ so that this implies \eqref{LH2YY}.
\end{proof}

\begin{proof}[Proof of \eqref{LH2}] Let $\eta=\phi(X_0)$. We have
$$ \L_{X_0^\vv}=\L_{X_0+\vv\phi(X_0)}=\mu\circ(\mathrm{Id}+\vv\phi)^{-1},\ \ \vv\in [0,1].$$
Moreover,     (\ref{t-Zeta}) with $\eta=\phi(X_0)$ implies  $\zeta^\eta_t=\zeta_t^\phi$ for $\zeta_t^\phi$ in \eqref{LH2}. So,
letting $k=0$ in \eqref{LH2YY} and taking expectation on both sides, we prove \eqref{LH2}.
 \end{proof}

\subsection{Proof of Theorem \ref{T3.1}(3)}
Let  $u^{\ll,\hat\mu}$ solve \eqref{PDE} for  $\hat\mu\in \C_{T,\scr P_2}$.
We first  characterize the Lipschitz continuity of $u^{\lambda,\mu}$ in $\mu$.

\begin{lem}\label{um}
Assume {\bf(H)} and let $T>0$. There exists a constant $c>0$ such that for any ${\hat\mu},{\hat\nu}\in C([0,T];\scr P_2)$,
\begin{align}\label{gth-1}
\|u_s^{\lambda,{\hat\mu}}-u_s^{\lambda,{\hat\nu}}\|_{\infty}+
\|\nabla u_s^{\lambda,{\hat\mu}}-\nabla u_s^{\lambda,{\hat\nu}}\|_{\infty}
\leq c\int_s^T \ff {\e^{-\ll (t-s)}}{\ss{t-s}}\W_2({\hat\mu}_t,{\hat\nu}_t)\d t,\ \ s\in [0,T].
\end{align}
 \end{lem}
\begin{proof}
 Let $P_{s,t}^{\hat\nu}$ be the Markov  semigroup associated with   the SDE:
$$\d X_{s,t}=\{B_t^{\hat\nu}+b_t^{\hat\nu}\}(X_{s,t})\d t+\sigma_t(X_{s,t})\d W_t, \ \ t\geq s\geq 0.$$
By \eqref{A1} with $B+b\equiv B^{\hat\nu}+b^{\hat\nu}$ and $\mu=\dd_x$ for $ x\in \R^d$, we find a constant $c_1>0$ such that
\beq\label{*02} \|\nn P_{s,t}^{\hat\nu} f\|_\infty\le \ff{c_1}{\ss{t-s}} \|f\|_\infty,\ \ f\in \B_b(\R^d), 0\le s <t\le T.\end{equation}
Next, by Duhamel's formula, the unique solution to \eqref{PDE} satisfies
$$
u_s^{\lambda,{\hat\mu}}=\int_s^T\e^{-\lambda(t-s)}P^{\hat\nu}_{s,t}(\nabla _{B_t^{\hat\mu}+b_t^{{\hat\mu}}-B_t^{\hat\nu}-b_t^{\hat\nu}}u_t^{\lambda,{\hat\mu}}+b_t^{\hat\mu})\d t,\ \ {\hat\mu}\in \C_{T,\scr P_2}.
$$
Moreover, by \eqref{LPS} and {\bf (H)}, we find a constant $c>0$ such that
$$\|B_t^{\hat\mu}+b_t^{{\hat\mu}}-B_t^{\hat\nu}-b_t^{\hat\nu}\|_\infty \le c \W_2(\hat\mu_t,\hat\nu_t),\ \ t\in [0,T].$$
Combining these with   \eqref{u0}  and \eqref{*02}, we find a constant $c_2>0$ such that
\begin{align*}
|u_s^{\lambda,{\hat\mu}}-u_s^{\lambda,{\hat\nu}}|
 &\le\int_s^T\big|\e^{-\lambda(t-s)}P^{\hat{\nu}}_{s,t}(\nabla _{B_t^{\hat\mu}+b_t^{{\hat\mu}}-B_t^{\hat\nu}-b_t^{\hat\nu}}u_t^{\lambda,{\hat\nu}}+b_t^{\hat\mu}-b_t^{\hat\nu})\big|\d t \\
 &\le c_2  \int_s^T\e^{-\lambda(t-s)} \W_2({\hat\mu}_t,{\hat\nu}_t)\d t,\\
|\nabla u_s^{\lambda,{\hat\mu}}-\nabla u_s^{\lambda,{\hat\nu}}|
\leq&  \int_s^T\left| \e^{-\lambda(t-s)} \nabla P^{\hat{\nu}}_{s,t} \big(\nabla _{B_t^{\hat\mu}+b_t^{{\hat\mu}}-B_t^{\hat\nu}-b_t^{\hat\nu}}u_t^{\lambda,{\hat\nu}}+b_t^{\hat\mu}-b_t^{\hat\nu}\big)\right|\d t\\
\leq&c_2  \int_s^T\frac{\e^{-\lambda(t-s)}}{\sqrt{t-s}} \W_2({\hat\mu}_t,{\hat\nu}_t)\d t,\ \ t\in [0,T].
\end{align*} Therefore, \eqref{gth-1} holds for some constant $c>0$.
\end{proof}

For $r\in[0,1], \mu\in \scr P_2, X_0\in L^2(\OO\to\R^d,\F_0;\P)$ and $\phi\in L^2(\R^d\to\R^d;\mu)$, let $X^{r}_t$ be the solution of (\ref{E1}) with initial value $X_0^{r}=X_0+r\phi(X_0)$, and denote
\beq\label{UUU}\mu_t^{r}:=\L_{X_t^{r}},\ \ t\in [0,T].\end{equation} We have
\beq\label{UR0} \mu_t^0=\L_{X_t},\ \ \mu_0^r= \L_{X_0+r\phi(X_0)}=\mu\circ(\mathrm{Id}+r\phi)^{-1},\ \ r\in[0,1], t\in [0,T].\end{equation}
 Let $\theta^{\ll,\mu^{r}}_t:=\mathrm{Id} + u_t^{\ll,\mu^{r}}$ for $u_t^{\ll,\mu^{r}}$ solving \eqref{PDE} with $ \mu^{r}$ replacing $\hat\mu$, and let
\beq\label{YYY}  Y^{r,\vv}_t:=\theta^{\lambda,\mu^{r}}_t(X^{r+\vv}_t),\ \ r,\vv\in [0,1], t\in [0,T].\end{equation}
Then
\beq\label{YXX} Y_t:= Y_t^{0,0}=\theta_t^{\ll,\mu}(X_t),\ \  \tilde{Y}_t^r:=Y^{r,0}_t=\theta^{\lambda,\mu^{r}}_t(X^{r}_t),\ \ r\in [0,1],t\in [0,T].\end{equation}
By \eqref{v-phi-eq} and \eqref{DXX} for the solution to \eqref{E1} with initial value $X_0+r\phi(X_0)$ and $\eta=\phi(X_0)$, we obtain
 \beq\label{DXX0} \nn_{\phi(X_0)} X_t^r = (\nn \theta_t^{\ll,\mu^r})^{-1}(X_t^r) v_t^{\phi,r},\ \ v_t^{\phi,r}:= \nn_{\phi(X_0)} \tilde{Y}_t^r,\end{equation}
 \beq\label{DXX01} \beg{split}
v_t^{\phi,r}
=&\,[\nn\theta^{\lambda,\mu^r}_0](X_0+r\phi(X_0))\phi(X_0)+\int_0^t\nn_{v_s^{\phi,r}}\tilde{\sigma}^{\mu^r}_s(\tilde{Y}_s^r)\,\d W_s\\
&+\int_0^t\Big\{\nn_{v_s^{\phi,r}}\tilde{b}^{\mu^r}_s(\tilde{Y}_s^r)+ F_s^{\mu^r}\big(X_s^r\big)v_s^{\phi,r} \Big\}\d s, \ \ t\in [0,T],\ r\in [0,1],\end{split}\end{equation}
where $F^{\mu^r}$ is defined in \eqref{FSM} with $\mu^r$ replacing $\mu$.
\beg{lem}\label{LN3}  Assume {\bf (H)}.
\beg{enumerate} \item[$(1)$]
 There exists a constant $c>0$ such that
\beq\label{Wp}  \sup_{r\in [0,1],t\in [0,T]} \W_2(\mu^{r}_t,\mu_t) \le    c   \|\phi\|_{L^2(\mu)},\ \ \phi\in L^2(\R^d\to\R^d;\mu).\end{equation}
Moreover, for any $p\ge 2$ there exists a constant $c(p)>0$ such that
\beq\label{CPP} \beg{split} & \sup_{r\in[0,1]}\E\Big(\sup_{t\in [0,T]}\big(|\tilde{Y}_t^r-Y_t|^p+|v_t^{\phi,r}|^p+|X_t^r-X_t|^p+|\nn_{\phi(X_0)}X_t^r|^p\big)\Big|\F_0\Big)\\
&\le c(p) \big( |\phi(X_0)|^p+\|\phi\|_{L^2(\mu)}^p\big),\ \ \phi\in L^2(\R^d\to\R^d;\mu).\end{split}\end{equation}
\item[$(2)$] Let $k\in[0,T)$ and $g\in C_b^1([0,T])$ with $g_k=0$ and $g_T=1$. For any $f\in \B_b(\R^d)$ and $\phi\in L^2(\R^d\to\R^d;\mu)$,
\beq\label{Wp2} \begin{split}\ff{\d}{\d r} \E (f(X_T^r)|\F_k)&=\lim_{\varepsilon\to 0}\frac{\E (f(X_T^{r+\varepsilon})-f(X_T^r)|\F_k)}{\varepsilon}\\
 &=\E\bigg(f (X_T^r) \int_k^T \<\zeta_t^{\phi,r}, \d W_t\>|\F_k\bigg),\ \ r\in [0,1] \end{split}\end{equation}holds for
\begin{align}\label{ZEpr}
\zeta_t^{\phi,r}= N_t(X_t^r) \nn_{\phi(X_0)}X_t^r,\ \ r\in [0,1], t\in [0,T],
\end{align}where for   $X,v\in L^2(\OO\to\R^d;\P)$,
\beq\label{NTT}
 N_t(X)v:= \si_t(X)^{-1}\Big\{ g_t' v+ g_t\E D^L (B_t+b_t)(y,\L_{X})(X)  v \big|_{y=X}\Big\},\ \ t\in [0,T].\end{equation}
Consequently, it holds
 \beq\label{Wp222} \begin{split}\ff{\d}{\d r} P_T f(\mu\circ (\mathrm{Id }+r\phi)^{-1})
 &=\E\bigg(f (X_T^r) \int_0^T \<\zeta_t^{\phi,r}, \d W_t\>\bigg),\ \ r\in [0,1]. \end{split}\end{equation}
\end{enumerate}
\end{lem}
\beg{proof}  (1) Recall that $Y_t^r=\theta_t^{\lambda,\mu}(X_t^r)$. Since $\W_2(\mu_t^r,\mu_t)^2\le \E|X_t^r-X_t|^2,$ \eqref{Wp} follows from \eqref{p2} and \eqref{u0}.
To prove \eqref{CPP}, let $\eta=\phi(X_0)$ and denote $\mu=\L_{X_0}$. By \eqref{YYY} for $\vv=0$, we have
\beq\label{*P*}X_t^{r}= (\theta_t^{\ll,\mu^{r}})^{-1}(\tilde{Y}_t^{r}),\ \ r\in [0,1], t\in [0,T],\end{equation} and by \eqref{xi-vv},
$$\vv \xi_t^\vv= Y_t^{\vv}-Y_t,\ \ t\in [0, T], \vv\in [0,1].$$
Then \eqref{xi-vv}, \eqref{gth-1}, \eqref{Wp}, \eqref{DXX01} and Lemma \ref{ti-v} with $\tilde{Y}_t^r$, $\mu^r$ replacing $Y_t$, $\mu$ respectively imply
\begin{align*}&\E\Big(\sup_{t\in [0,T]} \big\{|\tilde{Y}_t^r-Y_t|^p+|v_t^{\phi,r}|^p\big\}\Big|\F_0\Big)\\
&\leq c(p)\E\Big(\sup_{t\in [0,T]} \big\{|\tilde{Y}_t^r-Y_t^r|^p+|Y_t^r-Y_t|^p+|v_t^{\phi,r}|^p\big\}\Big|\F_0\Big)\\
&\le c(p) \big(|\phi(X_0)|^p+\|\phi\|_{L^2(\mu)}^p\big)
 \end{align*}
 for some constant $c(p)>0$. By \eqref{u0} and \eqref{*P*}, this implies \eqref{CPP}.

 (2)  By  \eqref{YYY} and  Lemma \ref{BYY} for $\tilde{Y}_t^r$ replacing $Y_t$ and $\mu^r$ replacing $\mu$, we obtain
\begin{equation}\label{D-X-s}\beg{split}
&\nabla_{\phi(X_0)} \E (f(\tilde{Y}^{r}_T)|\F_k):=\lim_{\vv\downarrow 0}\frac{\E (f(Y^{r,\vv}_T)-\E f(\tilde{Y}^{r}_T)|\F_k)}{\vv}\\
&=\E\left(f(\tilde{Y}^{r}_T)\int_k^T\<\zeta^{\phi,r}_t,\d W_t\>|\F_k\right),\ \ f\in\scr B_b(\R^d), r\in [0,1],k\in[0,T).
\end{split}\end{equation}
Combining this with \eqref{UUU}, \eqref{YYY} and \eqref{YXX}, we derive
\beg{align*} \ff{\d}{\d r} \E (f(X_T^r)|\F_k)
&= \lim_{\vv\downarrow 0}\ff{\E (f (X_T^{r+\vv}) - \E f  (X_T^{r})|\F_k)}\vv \\
&= \lim_{\vv\downarrow 0}\ff{\E ((f\circ(\theta_{T}^{\ll,\mu^r})^{-1})(Y_T^{r,\vv}) - \E (f\circ(\theta_{T}^{\ll,\mu^r})^{-1})(\tilde{Y}_T^{r})|\F_k)}\vv\\
&= \E\left(f(X^{r}_T)\int_k^T\<\zeta^{\phi,r}_t,\d W_t\>|\F_k\right).\end{align*}
In particular, for $k=0$, taking expectation on both sides of \eqref{Wp2}, we get \eqref{Wp222}.
Then the proof is finished.\end{proof}

Having  Lemmas \ref{um} and \ref{LN3} in hands, we  prove the $L$-differentiability of $P_Tf$ as follows by  modifying step (c) in the proof of \cite[Theorem 2.1]{RW}.

\beg{proof}[Proof of Theorem $\ref{T3.1}(3)$]  Let $\{\zeta_t^{\phi,r}\}_{r\in [0,1]}$ be in   Lemma \ref{LN3}. By {\bf (H)}, \eqref{CPP} and the Riesz representation theorem, there exists $\gg\in L^2(\R^d\to\R^d;\mu)$
such that
$$  \<\gg,\phi\>_{L^2(\mu)}= \E \bigg(f(X_T)\int_0^T \<\zeta_t^{\phi,0}, \d W_t\>\bigg),\ \ \phi\in L^2(\R^d\to\R^d;\mu).$$
By \eqref{Wp222} and \eqref{UR0}, we obtain
\beg{align*} &\ff{|P_Tf(\mu\circ(\mathrm{Id}+\phi)^{-1})- P_T f(\mu)-\<\gg, \phi\>_{L^2(\mu)}|}{\|\phi\|_{L^2(\mu)}} \\
&\le  \ff{1}{\|\phi\|_{L^2(\mu)} } \int_0^1\bigg|\E\left(f(X_T^r)\int_0^T\<\zeta_t^{\phi,r},\d W_t\>- f(X_T)\int_0^T\<\zeta_t^{\phi,0},\d W_t\>\right)\bigg|\d r\\
&\le \vv_1(\phi)+\vv_2(\phi)+\vv_3(\phi),\end{align*}
where, by \eqref{ZEpr} and \eqref{DXX0}
\beg{align*} &\vv_1(\phi):= \ff {1}{\|\phi\|_{L^2(\mu)}}  \int_0^1   \bigg|\E\bigg(\big(f(X_T^{r})-f(X_T)\big)\int_0^T \<  \zeta_t^{\phi,0}, \d W_t\>\bigg)\bigg|\d r,\\
&\vv_2(\phi):= \ff {\|f\|_\infty}{\|\phi\|_{L^2(\mu)}}  \int_0^1 \E \bigg|   \int_0^T \big\<\{N_t(X_t^{r})  - N_t(X_t)\}(\nn \theta_t^{\ll,\mu})^{-1}(X_t) v_t^\phi, \d W_t\big\>\bigg|\d r,\\
&\vv_3(\phi):= \ff {\|f\|_\infty}{\|\phi\|_{L^2(\mu)}}  \int_0^1 \E \bigg|   \int_0^T
\big\<N_t(X_t^{r}) \big\{(\nabla \theta_t^{\ll,\mu^r})^{-1}(X_t^r)v_t^{\phi,r} - (\nabla\theta_t^{\ll,\mu})^{-1}(X_t)v_t^\phi\big\}, \d W_t\big\>\bigg|\d r.
 \end{align*}
 So, it suffices to prove
$$\lim_{\|\phi\|_{L^2(\mu)}\to 0} \big\{\vv_1(\phi)+\vv_2(\phi)+\vv_3(\phi)\big\}=0.$$

a) We  first modify  the  proof of \cite[(2.3)]{RW}   to verify
\beq\label{WP*}\lim_{\|\phi\|_{L^2(\mu)}\to 0}  \vv_1(\phi)=0.\end{equation}
Denote
$$I_{k} := \int_0^k \big\<\zeta_t^{\phi,0},\d W_t\big\>,\ \    I_k^{\phi,r}=\big| \E[ I_k \{f(X_T^{r}) -f(X_T)\}]\big|, \ \
k\in (0,T), r\in (0,1].$$
By \eqref{Wp2} with $ g_t:=\ff{t-k}{T-k}$ for $t\in [k,T]$, {\bf (H)},  \eqref{DXX} and \eqref{CPP},  and noting that
$I_k$ is $\F_k$-measurable,
we   find     constants $c_1,c_2>0$ such that
\beg{align*} &I_k^{\phi,r} = \bigg|\E\bigg[I_k \int_0^r  \frac{\d}{\d\varepsilon}\E( f(X_T^{\vv})|\F_k)\d \vv\bigg]\bigg| \\
&\le \E\bigg[|I_k|\cdot\bigg|\int_0^r \E\bigg(f (X_T^\varepsilon) \int_k^T \<\zeta_t^{\phi,\varepsilon}, \d W_t\>\bigg|\F_k\bigg)\d \vv\bigg|\bigg]\\
&\le   c_1 \|f\|_\infty\int_0^r \E\bigg[|I_k|\bigg(\int_k^T \big\{\big| \nn_{\phi(X_0)}X_t^\vv\big|^2+\E|\nn_{\phi(X_0)}X_t^\vv|^2\big\} \d t\bigg)^{\ff 1 2}\bigg]\d \vv\\
&\le c_2 \|f\|_\infty \|\phi\|_{L^2(\mu)}^2,\ \ k\in [0,T),  \ f\in \B_b(\R^d).\end{align*}
So,
$$\lim_{\|\phi\|_{L^2(\mu)}\to 0} \ff 1 {\|\phi\|_{L^2(\mu)}} \int_0^1 I_k^{\phi,r}\d r=0,\ \ k\in (0,T).
$$
Combining this with  \eqref{CPP} and \eqref{ZEpr}, we obtain
  \beg{align*} & \lim_{\|\phi\|_{L^2(\mu)}\to 0} \vv_1(\phi)\le
  \lim_{\|\phi\|_{L^2(\mu)}\to 0}\ff {1}{\|\phi\|_{L^2(\mu)}}  \int_0^1   \bigg|\E\bigg(\big(f(X_T^{r})-f(X_T)\big)\int_k^T \<  \zeta_t^{\phi,0}, \d W_t\>\bigg)\bigg|\d r \\
  & \le c \|f\|_\infty\ss{T-k},\ \ k\in (0,T) \end{align*} for some constant $c>0$.
  By letting $k\uparrow T$ we prove \eqref{WP*}.

b) For any $\phi\in L^2(\R^d\to\R^d;\mu)$, $s\in [0,T]$ and $r\in [0,1],$ let
 $$h_{s,r}(\phi):= \big(\E|D^L(B_s+b_s)(y,\L_{X_s^r})(X_s^r)-D^L(B_s+b_s)(z,\L_{X_s})(X_s)|^2\big)^{\frac{1}{2}}\big|_{(y,z)=(X_s^r,X_s)}.$$
Noting that $h_{s,r}(\phi)^2<4(\|D^LB\|^2+\|D^Lb\|^2)_{T,\infty}$ due to {\bf (H)}, by {\bf (C)}, \eqref{CPP} and the dominated convergence theorem, we obtain
\beq\label{FSM''} \lim_{\|\phi\|_{L^2(\mu)}\to 0} h_{s,r}(\phi)=0,\ \ \sup_{(s,r)\in [0,T]\times [0,1]}\sup_{\|\phi\|_{L^2(\mu)}\le 1} h_{s,r}(\phi)^2<\infty.\end{equation}
Moreover, by    {\bf (H)} and \eqref{NTT}, we find a  constant $c_1>0$ such that
\beq\label{FSM'}\begin{split} &\|N_s(X_s^r)v- N_s(X_s)v\|\\
&\le c_1  \big\{|X_s^r-X_s|( |v|+(\E|v|^2)^{\frac{1}{2}}) + h_{s,r}(\phi) (\E|v|^2)^{\ff 1 2}\big\},
\ \ v\in L^2(\OO\to\R^d;\P).\end{split}\end{equation}
Combining \eqref{CPP}, \eqref{FSM''} and \eqref{FSM'}, we may find a constant $c_2>0$ such that   the dominated convergence theorem yields
\beg{align*} &\limsup_{\|\phi\|_{L^2(\mu)}\to 0} \vv_2(\phi) \le \limsup_{\|\phi\|_{L^2(\mu)}\to 0} \ff{c_2}{\|\phi\|_{L^2(\mu)}}
\int_0^1  \E\bigg(\int_0^T \big\{(|v_s^\phi|^2+\|\phi\|_{L^2(\mu)}^2)\cdot|X_s^r-X_s|^2\\
&\qquad\qquad\qquad\qquad\qquad\qquad\qquad\qquad\qquad\qquad\qquad+h_{s,r}(\phi)^2\|\phi\|_{L^2(\mu)}^2\big\}\d s \bigg)^{\ff 1 2}\d r\\
&\le \limsup_{\|\phi\|_{L^2(\mu)}\to 0}  \ff{c_2}{\|\phi\|_{L^2(\mu)}} \int_0^1\bigg\{ \Big(\E\Big[\sup_{s\in [0,T]}|v_s^\phi|^2\Big]+\|\phi\|_{L^2(\mu)}^2\Big)^{\ff 1 2}
\bigg(\E\int_0^T |X_s^r-X_s|^2\d s\bigg)^{\ff 1 2}\\
&\qquad\qquad\qquad \qquad\qquad\qquad \qquad\qquad\qquad  +\|\phi\|_{L^2(\mu)}\bigg(\int_0^T\E h_{s,r}(\phi)^2\d s\bigg)^{\ff 1 2}\bigg\}\d r=0.\end{align*}

c) It remains to prove
\beq\label{VV3} \lim_{\|\phi\|_{L^2(\mu)}\to 0}  \vv_3(\phi)=0.\end{equation}
Noting that
\begin{align*}
&|N_t(X_t^{r}) \big\{(\nabla \theta_t^{\ll,\mu^r})^{-1}(X_t^r)v_t^{\phi,r} - (\nabla\theta_t^{\ll,\mu})^{-1}(X_t)v_t^\phi\big\}|\\
&\leq c[\|(\nabla \theta_t^{\ll,\mu^r})^{-1}(X_t^r)-(\nabla\theta_t^{\ll,\mu})^{-1}(X_t)\||v_t^\phi|+|v_t^{\phi,r}-v_t^\phi|]\\
&+c\E [|D^L (B_t+b_t)(y,\L_{X_t^r})(X_t^r)|[\|(\nabla \theta_t^{\ll,\mu^r})^{-1}(X_t^r)-(\nabla\theta_t^{\ll,\mu})^{-1}(X_t)\|]|v_t^\phi|]  \big|_{y=X_t^r}\\
&+c\E [|D^L (B_t+b_t)(y,\L_{X_t^r})(X_t^r)||v_t^{\phi,r}-v_t^\phi|]  \big|_{y=X_t^r},
\end{align*}
according to BDG's inequality, \eqref{CPP} and \eqref{u0}, it suffices to prove
\begin{align*}
&\lim_{\|\phi\|_{L^2(\mu)}\to 0}\frac{\E\sup_{t\in[0,T]}\left\{\E \left[(1+|D^L (B+b)_t(y,\L_{X_t^r})(X_t^r)|)\sup_{t\in[0,T]}|v_t^{\phi,r}-v_t^\phi|\right]  \big|_{y=X_t^r}\right\}}{\|\phi\|_{L^2(\mu)}}\\
&=0.
\end{align*} By {\bf (C)}, this is implied by
\begin{align}\label{nee}
&\lim_{\|\phi\|_{L^2(\mu)}\to 0}\frac{\left(\E\sup_{t\in[0,T]}|v_t^{\phi,r}-v_t^\phi|^{\frac{p}{p-1}}\right)^{\frac{p-1}{p}} }{\|\phi\|_{L^2(\mu)}}=0.
\end{align} To prove \eqref{nee}, we observe that
\beg{align*}&|F_s^{\mu^r}(X_s^r)v_s^{\phi,r}- F_s^{\mu}(X_s)v_s^\phi|\\
&\leq |F_s^{\mu^r}(X_s^r)v_s^{\phi,r}- F_s^{\mu}(X_s^r)v_s^{\phi,r}|
+|F_s^{\mu}(X_s^r)v_s^{\phi,r}- F_s^{\mu}(X^r_s)v_s^{\phi}|
+|F_s^{\mu}(X^r_s)v_s^{\phi}- F_s^{\mu}(X_s)v_s^\phi|\\
&=:J_1+J_2+J_3.\end{align*}
By \eqref{FSM}, {\bf(H)}, Lemma \ref{um}, \eqref{u0} and \eqref{CPP},  we   find a constant $c_1>0$ such that
\beg{align*}
J_1 &\le \big|\nn(\theta_s^{\ll, \mu^r}-\theta_s^{\ll,\mu})(X_s^r)\E[D^L(B_s+b_s)(y, \L_{X_s^r})(X_s^r)(\nn\theta_s^{\ll,\mu})^{-1}(X_s^r)v_s^{\phi,r}]|_{y=X_s^r}\big|\\
&\leq c_1\|\phi\|_{L^2(\mu)}^2,\\
J_2&\leq c_1\E [|D^L (B_s+b_s)(y,\L_{X_s^r})(X_s^r)||v_s^{\phi,r}-v_s^\phi|]  \big|_{y=X_s^r}.
\end{align*}
 Moreover, similarly to \eqref{FSM'}, by {\bf(H)}, {\bf (C)}, \eqref{CPP} and \eqref{u0}, we find a   nonnegative random variables $\{h_{s,r}(\phi)\}_{s\in [0,T], r\in [0,1]}$ satisfying \eqref{FSM''} such that
\begin{align*}
&J_3\leq h_{s,r}(\phi) \|\phi\|_{L^2(\mu)}.
\end{align*}
So, there exists a constant $c_2>0$ such that
\beg{align*}&|F_s^{\mu^r}(X_s^r)v_s^{\phi,r}- F_s^{\mu}(X_s)v_s^\phi|\\
&\le c_2 \E [|D^L (B_s+b_s)(y,\L_{X_s^r})(X_s^r)||v_s^{\phi,r}-v_s^\phi|]  \big|_{y=X_s^r}\\
 &\qquad+c_2 (h_{s,r}(\phi)+\|\phi\|_{L^2(\mu)}) \|\phi\|_{L^2(\mu)},\ \ s\in [0,T], r\in [0,1].\end{align*}
Combining this with {\bf (H)}, \eqref{u0},  \eqref{gth-1}, \eqref{t-b-sig} and    \eqref{DXX01}, we find some constant $c_3>0$, a martingale $M_t^r$ with $\d\<M^r\>_t\le \d t$ and nonnegative random variables $\{h_{s,r}(\phi)\}_{s\in [0,T], r\in [0,1]}$ satisfying \eqref{FSM''} such that
\beq\label{SSS} \beg{split} &|v_t^{\phi,r}-v_t^\phi|\le c_3 |\phi(X_0)|(\|\phi\|_{L^2(\mu)}+[\nn\theta^{\lambda,\mu}_0](X_0+r\phi(X_0))-[\nn\theta^{\lambda,\mu}_0](X_0))\\
&+ c_3\int_0^t \big\{|v_s^{\phi,r}-v_s^\phi| +\E [|D^L (B_s+b_s)(y,\L_{X_s^r})(X_s^r)||v_s^{\phi,r}-v_s^\phi|]  \big|_{y=X_s^r}\big\}\d s\\
 &+ c_3\int_0^t \big\{h_{s,r}(\phi)(|v_s^\phi|+\|\phi\|_{L^2(\mu)})\big\}\d s \\
&+ c_3 \bigg|\int_0^t \big\{|v_s^{\phi,r}-v_s^\phi|+ h_{s,r}(\phi)|v_s^\phi|\big\} \d M_s^r\bigg|,\ \ t\in [0,T],r\in [0,1].\end{split}\end{equation}
By  BDG's inequality, H\"{o}lder's inequality, \eqref{CPP}, \eqref{u0} and \eqref{FSM''}, we find a constant $c_4>0$ and $\varepsilon(\phi)$ with $\lim_{\|\phi\|_{L^2(\mu)}\to 0}\varepsilon(\phi)=0$ such that
$$U_t:= \sup_{s\in [0,t]} |v_s^{\phi,r}-v_s^\phi|^{\frac{p}{p-1}},\ \ t\in [0,T]$$ satisfies
$$\E U_t \le \|\phi\|_{L^2(\mu)}^{\frac{p}{p-1}}\varepsilon(\phi) + c_4 \int_0^t\E U_s  \d s + \ff 1 2 \E U_t,\ \ t\in [0,T].$$
By Gronwall's lemma,  we obtain \eqref{nee} and the proof is completed.
\end{proof}

\end{document}